\documentclass[11pt]{amsart}

\usepackage{amsfonts}
\usepackage{amssymb}
\usepackage{amsmath}
\usepackage{amsthm}
\usepackage{stmaryrd} 
\usepackage{wasysym} 
\usepackage[colorlinks,allcolors=magenta]{hyperref}
\usepackage{tikz}

\usepackage[all]{xy}
\usepackage{enumitem}

\setenumerate{topsep=0pt,labelwidth=0pt,align=left,labelindent=0pt,labelsep=0in,labelwidth=0.275in,leftmargin=0.275in,label=\rm(\alph*),parsep=0pt,itemsep=1pt}

\renewcommand{\phi}{\varphi}

\renewcommand{\hat}{\widehat}

\renewcommand{\bar}{\overline}
\newcommand{\restr}{\upharpoonright}
\newcommand{\forces}{\mathrel{\Vdash}}

\newcommand{\iso}{\cong}

\newcommand{\Z}{\mathbb{Z}}
\newcommand{\N}{\mathbb{N}}
\newcommand{\Q}{\mathbb{Q}}
\newcommand{\R}{\mathbb{R}}
\newcommand{\C}{\mathbb{C}}

\newcommand{\B}{\mathbb{B}}
\newcommand{\A}{\mathbb{A}}

\renewcommand{\P}{\mathbb{P}}

\newcommand{\LF}{\mathcal{F}}
\newcommand{\LP}{\mathcal{P}}
\newcommand{\LC}{\mathcal{C}}

\newcommand{\LD}{\mathcal{D}}
\newcommand{\LG}{\mathcal{G}}
\newcommand{\LR}{\mathcal{R}}

\newtheorem{thm}{Theorem}[section]

\newtheorem{lemma}[thm]{Lemma}
\newtheorem{cor}[thm]{Corollary}
\newtheorem{ques}[thm]{Question}

\newtheorem*{marks_conj}{Marks's uniformity conjecture}

\theoremstyle{definition}
\newtheorem{defn}[thm]{Definition}
\newtheorem{example}[thm]{Example}

\theoremstyle{remark}

\newcommand{\ZFC}{\mathsf{ZFC}}

\renewcommand{\P}{\mathbb{P}}

\newcommand{\V}{\mathbf{V}}
\renewcommand{\L}{\mathbf{L}}

\newcommand{\Gen}{\mathrm{Gen}}

\newcommand{\Aut}{\mathrm{Aut}}

\newcommand{\SL}{\mathrm{SL}}

\title{Equivalence of generics}
\date{\today}
\author{Iian B. Smythe}
\address{Department of Mathematics, Rutgers, The State University of New Jersey, 110 Frelinghuysen Road, Piscataway, NJ, 08854}
\email{i.smythe@rutgers.edu}

\begin{document}

\begin{abstract}
Given a countable transitive model of set theory and a partial order contained in it, there is a natural countable Borel equivalence relation on generic filters over the model; two are equivalent if they yield the same generic extension. We examine the complexity of this equivalence relation for various partial orders, with particular focus on Cohen and random forcing. We prove, amongst other results, that the former is an increasing union of countably many hyperfinite Borel equivalence relations, while the latter is neither amenable nor treeable.
\end{abstract}

\maketitle

\section{Introduction}

Given a countable transitive model $M$ of $\ZFC$ and a partial order $\P$ in $M$, we can construct $M$-generic filters $G\subseteq\P$ and their corresponding generic extensions $M[G]$ using the method of forcing. We say that two such $M$-generic filters $G$ and $H$ are \emph{equivalent}, written $G\equiv^\P_M H$, if they produce the same generic extension, that is:
\[
	G\equiv^\P_M H \quad\text{if and only if}\quad M[G]=M[H].
\]
It is this equivalence relation that we aim to study.

The countability of $M$ and the definability of the forcing relation imply that $\equiv^\P_M$ is a \emph{countable Borel equivalence relation} (Lemma \ref{lem:equiv_CBER}), that is, each equivalence class is countable and $\equiv^\P_M$ is a Borel set of pairs in some appropriately defined space of $M$-generic filters for $\P$. The general theory of countable Borel equivalence relations affords us a broad set of tools for analyzing the relative complexity of each $\equiv^\P_M$; see the surveys \cite{MR1900547} and \cite{ThomasAST}. In turn, each $\equiv^\P_M$ provides a natural, well-motivated example.

To briefly review the general theory, for Borel equivalence relations $E$ and $F$ on Polish, or standard Borel, spaces $X$ and $Y$, a \emph{Borel reduction} of $E$ to $F$ is a Borel measurable function $f:X\to Y$ satisfying
\[
	x\,E\,y \quad\text{if and only if}\quad f(x)Ef(y)
\]
for all $x,y\in X$. If such an $f$ exists, we say that $E$ is \emph{Borel reducible} to $F$, written $E\leq_B F$. The relation $\leq_B$ gives a measure of complexity amongst Borel equivalence relations. If $E\leq_B F$ and $F\leq_B E$, then we say that $E$ and $F$ are \emph{Borel bireducible}; they have the same level of complexity. If $f$ only satisfies the forward implication in the displayed line above, we say that $f$ is a \emph{Borel homomorphism} of $E$ to $F$. A bijective Borel reduction from $E$ to $F$ is called a \emph{Borel isomorphism}, in which case we say that $E$ and $F$ are \emph{Borel isomorphic} and write $E\iso_B F$.
 
The simplest Borel equivalence relations, called \emph{smooth}, are those Borel reducible to the equality relation $\Delta(\R)$ on the reals. A benchmark example of a non-smooth countable Borel equivalence relation is eventual equality of binary strings, denoted by $E_0$:
\[
	x\,E_0\,y\quad\text{if and only if}\quad \exists m\forall n\geq m(x(n)=y(n)),
\]
for $x,y\in 2^\omega$.

Amongst the countable Borel equivalence relations, those Borel reducible to $E_0$ are exactly those which are \emph{hyperfinite}, that is, equal to an increasing union of countably many Borel equivalence relations, each with finite classes. These also coincide with orbit equivalence relations of Borel actions of $\Z$ (see Theorem 5.1 in \cite{MR1149121}).

Every hyperfinite equivalence relation is \emph{(Fr\'{e}chet) amenable}, see \cite{MR1900547} for the (somewhat technical) definition. In fact, every orbit equivalence relation induced by a countable amenable group is amenable (Proposition 2.13 in \cite{MR1900547}), while it remains open whether the converse holds, and whether amenability and hyperfiniteness coincide for equivalence relations.

More generally, every countable Borel equivalence relation can be realized as the orbit equivalence relation of a Borel action of some countable group (Theorem 1 in \cite{MR0578656}). Consequently, much of the theory consists of analyzing the dynamics of group actions. Of particular importance are the \emph{Bernoulli shift actions}: Given a countably infinite group $\Gamma$, $\Gamma$ acts on the space $2^\Gamma$ by:
\[
	(\gamma\cdot x)(\delta)=x(\gamma^{-1}\delta)
\]
for $x\in 2^\Gamma$, $\gamma,\delta\in\Gamma$. The corresponding orbit equivalence is denoted by $E(\Gamma,2)$. The \emph{free part} of this action,
\[
	(2)^\Gamma=\{x\in 2^\Gamma:\forall \gamma\in\Gamma(\gamma\neq 1\Rightarrow \gamma\cdot x\neq x)\},
\]
is a $\Gamma$-invariant Borel set on which the action is free, and is easily seen to be conull with respect to the usual product measure on $2^\Gamma$. Denote by $F(\Gamma,2)$ the restriction of $E(\Gamma,2)$ to $(2)^\Gamma$.

When $\Gamma=F_2$, the free group on $2$ generators, $E(F_2,2)$ and $F(F_2,2)$ are not amenable, and thus not hyperfinite (cf.~Proposition 1.7 in \cite{MR1900547} and \cite{MR960895}). $E(F_2,2)$ is \emph{universal}, every countable Borel equivalence relation is Borel reducible to it (Proposition 1.8 in \cite{MR1149121}). $F(F_2,2)$ is \emph{treeable}, meaning there is a Borel acyclic graph on the underlying space whose connected components are exactly the equivalence classes of $F(F_2,2)$. Every hyperfinite Borel equivalence relation is treeable, while no universal one is (see \cite{MR1900547}).
%

One precursor to the present work is the recent paper \cite{ClemCoskDwor} on the classification of countable models of $\ZFC$ up to isomorphism. While much of \cite{ClemCoskDwor} is concerned with ill-founded models, the proof of Theorem 3.2 therein, that $E_0$ Borel reduces to the isomorphism relation for countable well-founded models, makes essential use of the the fact that $\equiv^\C_M$, for $\C$ Cohen forcing, is not smooth. This observation was a starting point for our work here.

This paper is arranged as follows: Section \ref{sec:gen_results} consists of general results which apply to an arbitrary partial order $\P$ in $M$. We describe spaces of $M$-generic objects for $\P$, prove Borelness of $\equiv^\P_M$ on these spaces, verify that the Borel complexity of $\equiv^\P_M$ is independent of different presentations of $\P$, and discuss automorphisms of $\P$ and their relationship to $\equiv^\P_M$. We show that, for many of the partial orders one encounters in forcing, $\equiv^\P_M$ is not smooth (Theorem \ref{thm:hom_nonsmooth}).

Section \ref{sec:Cohen} is devoted to Cohen forcing $\C$. We prove that $\equiv^\C_M$ is an increasing union of countably many hyperfinite equivalence relations and is thus amenable (Theorem \ref{thm:cohen_hhf}).

In Section \ref{sec:random}, we consider random forcing $\B$. For groups $\Gamma\in M$, we establish a connection between $F(\Gamma,2)$ and $\equiv^\B_M$ (Theorem \ref{thm:randoms_shift}), and use this to show that $\equiv^\B_M$ is not amenable (Theorem \ref{thm:random_not_amen}), not treeable (Theorem \ref{thm:randoms_not_treeable}), and in particular, not hyperfinite. We also produce partial results concerning whether $\equiv^\B_M$ can be induced by a free action (Theorem \ref{thm:randoms_not_Mfree}) and whether it is universal (Theorem \ref{thm:randoms_not_univ}).

Section \ref{sec:questions} concludes the paper with a series of further questions which we hope will motivate continued study of the equivalence relations $\equiv^\P_M$.

\subsection*{Acknowledgements} I would like to thank Samuel Coskey, Joel David Hamkins, Andrew Marks, and Simon Thomas for many helpful conversations and correspondences.

\section{General results}\label{sec:gen_results}

Fix throughout a countable transitive model $M$ of $\ZFC$.\footnote{To avoid metamathematical concerns, one may, as always, work with a model of a large enough finite fragment of $\ZFC$. Nor do we really need $M$ to be countable; for a given partial order $\P$, it suffices that $\LP(\P)\cap M$ is countable (in $\V$). In particular, we could allow $M$ to be a transitive class in $\V$, such as when $\V$ is a generic extension of $M$ after sufficient collapsing, or $M=\L$ under large cardinal hypotheses. In such cases, only the proofs of Lemma \ref{lem:equiv_CBER} and \ref{lem:idealized_equiv_CBER} need alteration, instead relying on Theorem \ref{thm:bool_auts}.} 
When we assert that ``$\P$ is a partial order in $M$'', or just ``$\P\in M$'', we mean that $\P$ is a set partially ordered by $\leq$, and both $\P$ and $\leq$ are elements of $M$. $\P$ will always be assumed infinite, and thus, countably infinite in $\V$. These conventions also apply to Boolean algebras in $M$. We say that a Boolean algebra $\A$ is ``complete in $M$'' if $M\models\text{``$\A$ is a complete Boolean algebra''}$.

\subsection{Spaces of generics}

Recall that a filter $G\subseteq\P$ is \emph{$M$-generic} if it intersects each dense subset of $\P$ which is contained in $M$. Since $M$ is countable, such filters always exist (Lemma VII.2.3 in \cite{MR597342}).

\begin{defn}
	For a partial order $\P$ in $M$, the \emph{space of $M$-generics for $\P$}, denoted by $\Gen^\P_M$, is
	\[
		\Gen^\P_M = \{G\subseteq\P:G \text{ is an $M$-generic filter}\},
	\]
	We identify $\Gen^\P_M$ as a subspace of $2^\P$ with the product topology.		
\end{defn}

\begin{lemma}
	$\Gen^\P_M$ is a $G_\delta$ subset of $2^\P$, and thus a Polish space.
\end{lemma}

\begin{proof}
	Note that for any $p\in\P$, the sets $\{F\subseteq\P:p\in F\}$ and $\{F\subseteq\P:p\notin F\}$ are clopen. Enumerate the dense subsets of $\P$ in $M$ as $\{D_n:n\in\omega\}$. Then,
	\[
		\Gen^\P_M = \{F\subseteq\P:F\text{ is a filter}\}\cap\bigcap_{n\in\omega}\{F\subseteq\P:F\cap D_n\neq\emptyset\}.
	\]
	For each $n\in\N$,
	\[
		\{F\subseteq\P:F\cap D_n\neq\emptyset\}=\bigcup_{p\in D_n}\{F\subseteq\P:p\in D\},
	\]
	and these sets are open, so $\bigcap_{n\in\omega}\{F\subseteq\P:F\cap D_n\neq\emptyset\}$ is $G_\delta$.
	
	$F\subseteq\P$ is a filter if and only if $F\in\LF_1\cap\LF_2$, where
	\begin{align*}
		\LF_1&=\{F\subseteq\P:\forall p,q\in\P((p\in F\land p\leq q) \rightarrow q\in F)\},\\
		\LF_2&=\{F:\forall p,q((p,q\in F)\rightarrow\exists r(r\in F\land r\leq p,q))\}.
	\end{align*}
	Observe that
	\[
		\LF_1=\bigcap_{p\leq q\in\P}\left(\{F\subseteq\P:p\notin F\}\cup\{F\subseteq\P:q\in F\}\right),
	\]
	which is $G_\delta$, while
	\[
		\LF_2=\bigcap_{p,q\in\P}\left(\{F\subseteq\P:p\notin F\}\cup\{F\subseteq\P:q\notin F\}\cup\bigcup_{r\leq p,q}\{F\subseteq\P:r\in F\}\right)
	\]
	which is also $G_\delta$. Thus, $\Gen^\P_M$ is $G_\delta$.
\end{proof}

The following will simplify arguments involving the topology of $\Gen^\P_M$.

\begin{lemma}\label{lem:basic_clopens}
	The topology on $\Gen^\P_M$ has a basis consisting of clopen sets of the form
	\[
		N_p=\{G\in\Gen^\P_M:p\in G\}
	\]	
	for $p\in\P$.
\end{lemma}

\begin{proof}
	Since $\Gen^\P_M$ has the subspace topology it inherits from $2^\P$, the sets $N_p$ above are clopen. Suppose we are given a non-empty basic open set
	\[
		U=\{G\in \Gen^\P_M:p_0,\ldots,p_n\in G \text{ and }q_0,\ldots,q_m\notin G\},
	\]
	for $p_0,\ldots,p_n,q_0,\ldots,q_m\in\P$. Since $M$-generic filters for $\P$ are, in particular, maximal, we may find conditions $q_0',\ldots,q_m'\in\P$ so that
	\[
		U'=\{G\in \Gen^\P_M:p_0,\ldots,p_n,q_0',\ldots,q_m'\in G\}
	\]
	is non-empty and contained in $U$. Taking $p$ to be a common lower bound of $p_0,\ldots,p_n,q_0',\ldots,q_m'$, which exists since $U'\neq\emptyset$, we have $\emptyset\neq N_p\subseteq U'$.
\end{proof}

Recall that a partial order $\P$ is \emph{atomless} if for any $p\in\P$, there are $q,r\leq p$ with $q\perp r$ (i.e., they have no common lower bound). In all cases of interest, $\Gen^\P_M$ will be uncountable, a consequence of the following lemma.

\begin{lemma}\label{lem:atomless_perfect}
	$\Gen^\P_M$ has no isolated points if and only if $\P$ is atomless.
\end{lemma}

\begin{proof}
	($\Leftarrow$) Suppose that $\P$ is atomless. Take $G\in\Gen^\P_M$. By Lemma \ref{lem:basic_clopens}, it suffices to consider a basic open set $N_p$ containing $G$, for $p\in\P$. Since $\P$ is atomless, there are $q,r\leq p$ with $q\perp r$. Take $G'$ and $G''$ to be $M$-generic filters containing $q$ and $r$, respectively. Then $G',G''\in N_p$, and at least one of them must be unequal to $G$, showing that $G$ is not isolated.
	
	($\Rightarrow$) Suppose that $\Gen^\P_M$ has no isolated points. Take $p\in\P$, and let $G$ be an $M$-generic filter containing $p$. Since $\Gen^\P_M$ has no isolated points, there is a filter $G'\in N_p$ distinct from $G$, say with $p'\in G'\setminus G$. As $G$ is maximal, there is a $q\in G$, which we may assume is $\leq p$, with $q\perp p'$. But $p\in G'$, so there is also a $r\leq p,p'$, and thus $q\perp r$.
\end{proof}

\begin{defn}
	For $\P$ a partial order in $M$, define $\equiv^\P_M$ on $\Gen^\P_M$ by
\[
	G \equiv^\P_M H \quad\text{if and only if}\quad M[G]=M[H].
\]
\end{defn}

By the minimality of generic extensions (Lemma VII.2.9 in \cite{MR597342}), $G \equiv^\P_M H$ if and only if $G\in M[H]$ and $H\in M[G]$.

\begin{lemma}\label{lem:equiv_CBER}
	$\equiv^\P_M$ is a countable Borel equivalence relation.
\end{lemma}

\begin{proof}
	For $G\in\Gen^\P_M$, the $\equiv^\P_M$-class of $G$ is a subset of $M[G]$, which is countable as there are only countably many $\P$-names in $M$. To see that $\equiv^\P_M$ is Borel, note that $G\in M[H]$ if and only if there is a $\P$-name $\tau\in M$ such that for every $p\in\P$,
	\[
		p\in G \quad\text{if and only if}\quad \exists q\in H(q\forces \check{p}\in\tau).
	\]
	Since the forcing relation is arithmetic in a code for the model $M$, and both $M$ and $\P$ are countable, this is a Borel condition on $G$ and $H$.
\end{proof}

The next lemma verifies that the Borel complexity of $\equiv^\P_M$ is invariant for forcing-equivalent presentations of $\P$.

\begin{lemma}\label{lem:equiv_posets}
	Suppose that $\P$ and $\Q$ are partial orders in $M$.
	\begin{enumerate}
		\item If $i:\P\to\Q$ is a dense embedding in $M$, then $\equiv^\P_M\;\iso_B\;\equiv^\Q_M$.
		\item If $\P$ and $\Q$ have isomorphic Boolean completions in $M$, then $\equiv^\P_M\;\iso_B\;\equiv^\Q_M$.
	\end{enumerate}
\end{lemma}

\begin{proof}
	(a) Define $\hat{i}:\Gen^\Q_M\to\Gen^\P_M$ by $\hat{i}(H)=i^{-1}(H)$. By standard results (Theorem VII.7.11 in \cite{MR597342}), $\hat{i}$ is a well-defined bijection with inverse
	\[
		\hat{i}^{-1}(G)=\{q\in\Q:\exists p\in G(i(p)\leq q)\},
	\]
	and satisfying $M[H]=M[\hat{i}(H)]$ for all $H\in\Gen^\Q_M$. Thus
	\[
		M[H_0]=M[H_1] \quad\text{if and only if}\quad M[\hat{i}(H_0)]=M[\hat{i}(H_1)]
	\]
	for all $H_0,H_1\in\Gen^\Q_M$, showing that $\hat{i}$ is a reduction. To see that $\hat{i}$ is Borel, given $N_p\subseteq\Gen^\P_M$, for $p\in\P$, a basic clopen as in Lemma \ref{lem:basic_clopens},
	\begin{align*}
		\hat{i}^{-1}(N_p)&=\{H\in\Gen^\Q_M:p\in\hat{i}(H)\}\\
		&=\{H\in\Gen^\Q_M:i(p)\in H\},
	\end{align*}
	which is clopen in $\Gen^\Q_M$. Thus, $\hat{i}$ is continuous, and in particular, Borel.
		
	(b) Follows immediately from (a) by composing Borel isomorphisms.
\end{proof}

Every Borel set in $M$ is \emph{coded} by a real $\alpha$ in $M$; we denote the interpretation of this code in a model $N\supseteq M$ by $B_\alpha^N$, a Borel set in $N$, omitting the superscript when $N=\V$. We will often just refer to a Borel set $B$ coded in $M$, without reference to the code itself. Likewise for Borel functions coded in $M$, identified with their Borel graphs. The basic properties of Borel codes can be found in \cite{MR1940513} or \cite{MR0265151}. 

Many partial orders whose generic extensions are generated by adjoining a single real can be presented as \emph{idealized forcings} \cite{MR2391923}, that is, as the set $P_I$ of all Borel subsets of $2^\omega$ not in $I$, ordered by inclusion, where $I$ is a non-trivial $\sigma$-ideal of Borel sets.

Given an $M$-generic filter $G$ for an idealized forcing $P_I$ in $M$, there is a unique (Proposition 2.1.2 in \cite{MR2391923}) real $x_G\in 2^\omega$ in $M[G]$ such that
\[
	\{x_G\}=\bigcap\{B^{M[G]}:B^M\in G\}=\bigcap\{B:B^M\in G\},
\]
called an \emph{$M$-generic real} for $P_I$. Since $G$ is computed from $x_G$ in any model $N\supseteq M$ containing it as
\[
	G=\{B^M\in P_I:x_G\in B^N\},
\] 
we have that $M[G]=M[x_G]$.

Given an idealized forcing $P_I$ in $M$, let $\LG^{P_I}_M$ be the set of all $M$-generic reals for $P_I$. Abusing notation, we define $\equiv^{P_I}_M$ in the obvious way on $\LG^{P_I}_M$.

\begin{lemma}\label{lem:idealized_Borel}
	$\LG^{P_I}_M$ is a Borel subset of $2^\omega$, and thus a standard Borel space.	
\end{lemma}

\begin{proof}
	$x$ is an $M$-generic real for $P_I$ if and only if
	\[
		x\in\bigcap_{{\mathcal{D}\in M}\atop\text{$\mathcal{D}$ dense in $P_I$}}\bigcup\{B:B^M\in\LD\}.
	\]
	Since $M$ is a countable transitive model, the set on the right is Borel.
\end{proof}

Exactly as in Lemma \ref{lem:equiv_CBER}, we have:

\begin{lemma}\label{lem:idealized_equiv_CBER}
	$\equiv^{P_I}_M$ is a countable Borel equivalence relation.\qed
\end{lemma}

\begin{lemma}\label{lem:idealized_equiv_posets}
	Suppose that in $M$, $\P$ is a partial order, $I$ a $\sigma$-ideal of Borel subsets of $2^\omega$, and $i:P_I\to\P$ a dense embedding. Then, $\equiv^{P_I}_M\;\iso_B\;\equiv^\P_M$.
\end{lemma}

\begin{proof}
	As in the proof of Lemma \ref{lem:equiv_posets}, the map $\hat{i}:\Gen^\P_M\to\LG^{P_I}_M$ given by $\hat{i}(H)=x_{i^{-1}(H)}$, the $M$-generic real corresponding to $i^{-1}(H)$, is a bijection which maps $\equiv^\P_M$-equivalent filters to $\equiv^{P_I}_M$-equivalent reals, and vice-versa. To see that $\hat{i}$ is Borel, given $N_p\subseteq\Gen^\P_M$ with $p\in\P$,
	\begin{align*}
		\hat{i}(N_p)=\{x\in\LG^{P_I}_M:x\in i(p)^\V\}=\LG^{P_I}_M\cap i(p)^\V,
	\end{align*}
	which is Borel by Lemma \ref{lem:idealized_Borel}. Thus, $\hat{i}^{-1}$ is Borel, and so $\hat{i}$ is as well.
\end{proof}

In sum, when analyzing the complexity of $\equiv^\P_M$ we will be able to use various equivalent presentations of $\P$.

\subsection{Automorphisms and homogeneity}


Given a partial order $\P$ in $M$, let $\Aut^M(\P)$ be the automorphism group of $\P$ in $M$. By absoluteness,
\[
	\Aut^M(\P)=\Aut(\P)\cap M.
\]

There is a natural action $\Aut^M(\P)\curvearrowright
\Gen^\P_M$ given by
\[
	(e, G)\mapsto e''G=\{e(p):p\in G\},
\]
for $e\in\Aut^M(\P)$ and $G\in\Gen^\P_M$. This action is well-defined and $\equiv^\P_M$-invariant, in the sense that $M[e'' G]=M[G]$ (Corollary VII.7.6 in \cite{MR597342}). Treating $\Aut^M(\P)$ as a discrete group, this action is continuous.

A partial order $\P$ is \emph{weakly homogeneous} (in $M$) if for all $p,q\in\P$, there is an automorphism $e$ of $\P$ (in $M$) such that $e(p)$ is compatible with $q$.

A continuous action of a group $\Gamma$ (or equivalence relation $E$, respectively) on a Polish space $X$ is \emph{generically ergodic}  if every $\Gamma$-invariant ($E$-invariant, respectively) Borel set is either meager or comeager. Equivalently, the action of $\Gamma$ is generically ergodic if and only if for all non-empty open $U,V\subseteq X$, there is a $\gamma\in\Gamma$ such that $\gamma(U)\cap V\neq\emptyset$ (cf.~Proposition 6.1.9 in \cite{MR2455198})

\begin{lemma}\label{lem:gen_ergodic}
	Let $\P$ be a partial order in $M$. The action of $\Aut^M(\P)$ on $\Gen^\P_M$ is generically ergodic if and only if $\P$ is weakly homogeneous in $M$.
\end{lemma}

\begin{proof}
	($\Leftarrow$) Suppose that $\P$ is weakly homogeneous. Let $U,V\subseteq\Gen^\P_M$ be non-empty open sets. By Lemma \ref{lem:basic_clopens}, we may assume that $U=N_p$ and $V=N_q$, for some $p,q\in\P$. By weak homogeneity, there is an $e\in\Aut^M(\P)$ such that $e(p)$ is compatible with $q$, say with common lower bound $r$. Let $G$ be an $M$-generic filter for $\P$ which contains $e^{-1}(r)$. Since $e^{-1}(r)\leq p$, $G\in N_p$. Likewise, since $r\in e''G$ and $r\leq q$, $e''G\in N_q$, proving $e(U)\cap V\neq\emptyset$.
	
	($\Rightarrow$) Suppose that the action is generically ergodic. Pick $p,q\in\P$. By generic ergodicity, there is an $e\in\Aut^M(\P)$ such that $e(N_p)\cap N_q\neq\emptyset$. Say $e''G\in e(N_p)\cap N_q$. But then, $e(p),q\in e''G$, and $e''G$ is a filter, so $e(p)$ and $q$ are compatible.
\end{proof}

\begin{thm}\label{thm:hom_nonsmooth}
	Let $\P$ be an atomless partial order in $M$. If $\P$ is weakly homogeneous in $M$, then $\equiv^\P_M$ is not smooth.
\end{thm}

\begin{proof}
	Since $\P$ is atomless, countable sets are meager in $\Gen^\P_M$ by Lemma \ref{lem:atomless_perfect}. The action of $\Aut^M(\P)$ of $\Gen^\P_M$ is generically ergodic by Lemma \ref{lem:gen_ergodic}, and has meager orbits, so it is not smooth (Proposition 6.1.10 in \cite{MR2455198}). The induced orbit equivalence relation is a subequivalence relation of $\equiv^\P_M$, and thus these properties are inherited by $\equiv^\P_M$.
\end{proof}

We caution that weak homogeneity is not preserved by forcing equivalence; for every partial order, there is a \emph{rigid} (i.e.,~having trivial automorphism group) partial order with the same Boolean completion \cite{MO184806}. However, weak homogeneity of the Boolean completion is preserved by forcing equivalence (cf.~Theorem 8 in \cite{MR3490908}). Moreover, the apparently weaker condition of \emph{cone homogeneity}, that for every $p,q\in\P$, there are $p'\leq p$ and $q'\leq q$ such that $\P\restr p'\iso\P\restr q'$, implies weak homogeneity of the Boolean completion provided $\P$ is atomless (Fact 1 in \cite{MR2448954}). In practice, this latter property is often easier to verify.

The following are examples of atomless partial orders $\P$ satisfying one of the aforementioned homogeneity conditions, and thus, by Theorem \ref{thm:hom_nonsmooth}, yield a non-smooth $\equiv^\P_M$.

\begin{example}
	\emph{Cohen forcing}: If we take our presentation of Cohen forcing $\C$ to be the infinite binary tree $2^{<\omega}$, ordered by extension, then given any $p,q\in 2^{<\omega}$, say with $|p|\leq|q|$, we can extend $p$ to $p'$ with $|p'|=|q|$, and use that the automorphism group of $2^{<\omega}$ acts transitively on each level to get an automorphism $e$ such that $e(p)\leq e(p')=q$. As mentioned in the introduction, the non-smoothness of $\equiv^\C_M$ was previously observed in \cite{ClemCoskDwor}.
\end{example}

\begin{example}
	\emph{Random forcing}: Let $\B$ be all non-null Borel subsets of $2^\omega$ in $M$, ordered by an inclusion. This is the idealized forcing corresponding the null ideal. The weak homogeneity of $\B$ boils down to the following fact: Whenever $A$ and $B$ are positive measure Borel sets, there is an $s\in 2^{<\omega}$ such that the translate of $A$ by $s$, adding modulo $2$ in each coordinate, has positive measure intersection with $B$. This is a consequence of the Lebesgue Density Theorem (cf.~Chapter 7 of \cite{MR924157}): Take basic open sets determined by finite binary strings of the same length and having a sufficiently large proportion of their mass intersecting each of $A$ and $B$, respectively. Then, the $s\in 2^{<\omega}$ which translates one string to the other will be as desired.
\end{example}

\begin{example}\label{ex:other_posets}
	\emph{Sacks forcing}, \emph{Miller forcing}, \emph{Mathias forcing}, \emph{Laver forcing} and many other classical forcing notions are easily seen to be cone homogeneous for the following reason: the cone below any condition is isomorphic to the whole partial order. For descriptions of these examples, see \cite{MR1940513}.
\end{example}

Being a countable Borel equivalence relation, $\equiv^\P_M$ is induced by some Borel action of a countable group. When the partial order is a complete Boolean algebra $\A$ in $M$, this group can be taken to be $\Aut^M(\A)$, and the action the canonical one on generics described above:

\begin{thm}[Vop\v{e}nka--H\'{a}jek \cite{MR0444473}; Theorem 3.5.1 in \cite{MR0373889}]\label{thm:bool_auts}
	Let $\A$ be a complete Boolean algebra in $M$. If $G$ and $H$ are $M$-generic filters for $\A$ and $M[G]=M[H]$, then there is involutive automorphism $e$ of $\A$ in $M$ such that $H=e''G$.	
\end{thm}

\begin{example}\label{ex:rigid}
	As an immediate consequence of Theorem \ref{thm:bool_auts}, if $\A$ is a rigid complete Boolean algebra in $M$, then $\equiv^\A_M$ is smooth. An example of a non-trivial Boolean algebra with this property appears in \cite{MR0292670}. Another example, which shares properties (e.g., fusion and adding reals of minimal degree) with Sacks and Miller forcing, is given in \cite{MR1398120}.
\end{example}

\section{Cohen reals}\label{sec:Cohen}

To further analyze the complexity of $\equiv^\C_M$ for Cohen forcing $\C$, we will find it useful to consider its idealized presentation; let $\C$ be all non-meager Borel subsets of $2^\omega$, ordered by containment, as computed in $M$. Since $\C$ satisfies the countable chain condition, we can identify the Boolean completion $\bar{\C}$ of $\C$, in $M$, with the quotient of all Borel subsets modulo the meager ideal (cf.~Lemma II.2.6.3 in \cite{MR0265151}). Let $\LC$ be the set of $M$-generic Cohen reals in $2^\omega$, that is, $\LC=\LG^\C_M$ in the notation of \S\ref{sec:gen_results}. The main result of this section is the following:

\begin{thm}\label{thm:cohen_hhf}
	$\equiv^\C_M$ is an increasing union of countably many hyperfinite Borel equivalence relations. In particular, $\equiv^\C_M$ is amenable.
\end{thm}

We will need the following results from the literature:

\begin{thm}[Hjorth--Kechris \cite{MR1423420}, Sullivan--Weiss--Wright \cite{MR833710}, Woodin; Theorem 12.1 of \cite{MR2095154}]\label{thm:gen_hf}
	If $E$ is a countable Borel equivalence relation on a Polish space $X$, then there is a comeager Borel set $C\subseteq X$ such that $E\restr C$ is hyperfinite.
\end{thm}

\begin{thm}[Maharam--Stone \cite{MR522175}]\label{thm:realize_auts}
	Every automorphism of the Boolean algebra of Borel sets modulo meager sets, in a complete metric space $X$, is induced by a meager-class preserving Borel bijection $f:X\to X$.
\end{thm}

Here, a bijection $f:X\to X$ on a Polish space is said to be \emph{meager-class preserving} if both $f$ and $f^{-1}$ preserve meager sets.

A property of a Borel set $B$ coded in $M$ is \emph{absolute} if it has the same truth value for $B^M$ in $M$ as for $B$ in $\V$. We can show that a property of Borel sets is absolute by expressing it as a $\Pi^1_1$ predicate in their codes; absoluteness then follows by Mostowski's Absoluteness Theorem (Theorem 25.4 in \cite{MR1940513}). A property of $B$ is \emph{upwards absolute} if its truth for $B^M$ in $M$ implies its truth for $B$ in $\V$.

Coding a Borel set, (non-)membership in a Borel set, and containment of Borel sets, are all expressible as $\Pi^1_1$ predicates in the codes (Theorem II.1.2 and Corollary II.1.2 in \cite{MR0265151}), and thus absolute. These will be used implicitly in what follows. We will need the absoluteness of several additional properties:

\begin{lemma}\label{lem:inj_borel_abs}
	The following notions are absolute for Borel sets $A$, $B$, and injective Borel functions $f$, coded in $M$:
	\begin{enumerate}[label=\textup{(\roman*)}]
		\item $f$ is an injective function.
		\item $A=f''B$.
	\end{enumerate}
\end{lemma}

\begin{proof}
	(i) The statements that ``$f$ is a function'' and that ``$f$ is injective'' are universal in predicates for membership and non-membership in the graph of $f$, and thus can be expressed as $\Pi^1_1$ predicates in a code for $f$.

	(ii) It suffices to show that ``$x\in f''B$'' is $\Pi^1_1$ in the  codes for $f$ and $B$. This is analogous to the classical fact that injective Borel images of Borel sets are Borel and can be derived from its effective version (4D.7 in \cite{MR2526093}).
\end{proof}

\begin{lemma}\label{lem:eqrel_abs}
	The following notions are absolute for a Borel equivalence relation $E$ coded in $M$:
	\begin{enumerate}[label=\textup{(\roman*)}]
		\item $E$ is an equivalence relation.
		\item $E$ has finite classes.
		\item For a countable group $\Gamma$ in $M$ with a Borel action on $2^\omega$ coded in $M$, $E$ is the induced orbit equivalence relation.
	\end{enumerate}
\end{lemma}

\begin{proof}
	(i) The statement that ``$E$ is an equivalence relation'' is universal in predicates for membership and non-membership in $E$, and thus can be expressed as a $\Pi^1_1$ predicate in a code for $E$.
	
	
	(ii) Being an equivalence class of a Borel equivalence relation can be expressed as a $\Pi^1_1$ predicate using predicates for membership and non-membership in the equivalence relation. Since the property of being finite can be expressed using a $\Pi^1_1$ predicate (cf.~Lemma II.1.6.7 in \cite{MR0265151}), having all classes finite is expressible as a $\Pi^1_1$ predicate as well.
	
	(iii) Both the statement that a Borel function $f:\Gamma\times 2^\omega\to 2^\omega$ is an action of $\Gamma$, and the fact that $E$ is the resulting orbit equivalence relation, are easily expressible as $\Pi^1_1$ predicates in the appropriate codes; we leave the verification to the reader.
\end{proof}

\begin{lemma}\label{lem:hf_up_abs}
	Being a hyperfinite Borel equivalence relation is upwards absolute\footnote{Being hyperfinite is a $\Sigma^1_2$ property in the codes, but whether it is $\Sigma^1_2$-complete, and thus not absolute for countable models, appears to be unknown.} for a Borel equivalence relation coded in $M$.	
\end{lemma}

\begin{proof}
	Suppose $E$ is a Borel equivalence relation coded in $M$, and that in $M$, $E^M=\bigcup_{n\in\omega} E_n^M$, where each $E_n^M$ is a Borel equivalence relation with finite classes. Since containment and countable unions are absolute for Borel codes (Lemma II.1.6.1 in \cite{MR0265151}), $E=\bigcup_{n\in\omega} E_n$ in $\V$, and by Lemma \ref{lem:eqrel_abs}, is hyperfinite.
\end{proof}

If $g$ is a meager-class preserving Borel bijection coded in $M$, then by the absoluteness of meagerness (Lemma II.1.6.6 in \cite{MR0265151}) and Lemma \ref{lem:inj_borel_abs} above, $g^M$ is a meager-class preserving Borel bijection in $M$, and induces an automorphism of $\C$ (and $\bar{\C}$) given by: 
	\[
		B\mapsto (g^M)''(B).
	\]
	This in turn induces an action on $M$-generic filters for $\C$,  	\[
		H\mapsto g^M\cdot H=\{(g^M)''B:B\in H\}.
	\] 
	Passing to the generic reals, this action coincides with $g$'s action on $2^\omega$:

\begin{lemma}\label{lem:map_Cgens}
	Let $g:2^\omega\to 2^\omega$ be a meager-class preserving Borel bijection coded in $M$, $x\in\LC$, and $H_x$ the corresponding $M$-generic filter for $\C$. Then,
	\[
		(g^M)\cdot H_x=H_{g(x)}.
	\]
\end{lemma}

\begin{proof}
	Given the correspondence between $M$-generic Cohen reals and $M$-generic filters for $\C$, it suffices to prove that
	\[
		\bigcap\{A:A^M\in g^M\cdot H_x\}=\{g(x)\}.
	\]
	Let $B^M\in H_x$. Then, $x\in B$, so $g(x)\in g''B$. Let $A^M=(g^M)''B$. By Lemma \ref{lem:inj_borel_abs}(ii), $A=g''B$, and thus $g(x)\in A$. This shows that
	\[
		g(x)\in\bigcap\{A:A^M\in g^M\cdot H_x\}.
	\]
	This intersection is a singleton by the genericity of $g^M\cdot H_x$.
\end{proof}

\begin{proof}[Proof of Theorem \ref{thm:cohen_hhf}]
	In $\V$, enumerate $\Aut^M(\bar{\C})$ as $\{\gamma_n:n\in\omega\}$. 
	
	For each $n$, apply Theorem \ref{thm:realize_auts} in $M$ to find a meager-class preserving Borel bijection $f_n^M$ of $(2^\omega)^M$ which induces the automorphism $\gamma_n$ on $\bar{\C}$. Let $\Gamma_n$ be an abstract group in $M$ which is isomorphic to that generated by $\{f_0^M,\ldots,f_{n-1}^M\}$. We use $\Gamma_n$, rather than the group generated by $\{f_0^M,\ldots,f_{n-1}^M\}$, to simplify matters when passing to $\V$. This induces a Borel action of $\Gamma_n$ on $(2^\omega)^M$; let $E_n^M$ be the induced orbit equivalence relation. Note that while each $f_n^M$ and $\Gamma_n$ is in $M$, their enumeration is not.
		
	For each $n$, apply Theorem \ref{thm:gen_hf} in $M$ to obtain a comeager Borel set $C_n^M$ on which $E_n^M$ is hyperfinite.
	
	In $\V$, let $C=\bigcap_{n\in\omega} C_n$, which is comeager by the absoluteness of meagerness (Lemma II.1.6.6 in \cite{MR0265151}), and let $E=\bigcup_{n\in\omega} (E_n\restr C)$. By Lemma \ref{lem:hf_up_abs}, $E$ is an increasing union of hyperfinite Borel equivalence relations. Since a real is in $\LC$ if and only if it is contained in every comeager Borel set coded in $M$ (cf.~\S II.2 of \cite{MR0265151}), $\LC\subseteq C$.
	
	We claim that $\equiv^\C_M$ coincides with $E\restr \LC$. Suppose that $x,y\in \LC$ are $E$-related. Then, there is some $n$ for which $xE_n y$. By Lemma \ref{lem:eqrel_abs}(iii), $\Gamma_n$ induces $E_n$, so there is a $g\in \Gamma_n$ with $y=g\cdot x$. $g$'s action on $2^\omega$ can be expressed as a word in $f_0^{\pm1},\ldots,f_{n-1}^{\pm1}$. Let $g^M$ be the corresponding word in $(f_0^M)^{\pm1},\ldots,(f_{n-1}^M)^{\pm1}$. In $M$, $g^M$ is a meager-class preserving Borel bijection $(2^\omega)^M\to(2^\omega)^M$, and thus, induces an automorphism of $\C$ (in $M$) mapping the $M$-generic filter corresponding to $x$ to that corresponding to $y$, by Lemma \ref{lem:map_Cgens}. Hence, $M[x]=M[y]$.
	
	Conversely, suppose $x,y\in \LC$ are $\equiv^\C_M$-related. Since $\bar{\C}$ is a complete Boolean algebra in $M$, we can apply Theorem \ref{thm:bool_auts} to obtain a $\gamma_k\in\Aut^M(\bar{\C})$ such that $\gamma_k$ maps the $M$-generic filter corresponding to $x$ to that corresponding $y$. By Lemma \ref{lem:map_Cgens}, $f_k(x)=y$, and so $x E_{k+1} y$.
	
	The amenability of $\equiv^\C_M$ follows from the fact that hyperfinite equivalence relations are amenable, and increasing unions of countably many amenable equivalence relations are amenable (Propositions 2.15 in \cite{MR1900547}).
\end{proof}

%

\section{Random reals}\label{sec:random}

As in \S\ref{sec:Cohen}, we will view \emph{random forcing} in its idealized form, as the set $\B$ of all non-null Borel subsets of $2^\omega$, ordered by containment, as computed in $M$. Since $\B$ also satisfies the countable chain condition, we can identify its completion $\bar{\B}$ in $M$ with the quotient of all Borel sets modulo the null ideal. Let $\LR$ be the set of $M$-generic random reals in $2^\omega$.


As in \S\ref{sec:Cohen}, if $g$ is a measure-preserving Borel bijection of $2^\omega$ coded in $M$, then by the absoluteness of Lebesgue measure (Lemma II.1.6.4 in \cite{MR0265151}), $g^M$ induces an automorphism of $\B$ (and $\bar{\B}$) by $B\mapsto (g^M)''B$. This induces an action on $M$-generic filters for $\B$ by $H\mapsto g^M\cdot H=\{(g^M)''B:B\in H\}$. The following is the analogue of Lemma \ref{lem:map_Cgens} for $\B$ and its proof is identical.

\begin{lemma}\label{lem:map_Bgens}
	Let $g:2^\omega\to 2^\omega$ be a measure-preserving Borel bijection coded in $M$, $x\in\LR$, and $H_x$ the corresponding $M$-generic filter for $\B$. Then,
	\[
		(g^M)\cdot H_x=H_{g(x)}.
	\]
	\qed
\end{lemma}

\begin{lemma}\label{lem:Bern_shift_abs}
	The following notions are absolute for a countably infinite group $\Gamma$ in $M$:
	\begin{enumerate}[label=\textup{(\roman*)}]
		\item The Bernoulli shift action of $\Gamma$.
		\item The free part of the Bernoulli shift action of $\Gamma$.
	\end{enumerate}
\end{lemma}

\begin{proof}
	(i) Let $s^M:\Gamma\times (2^\Gamma)^M\to(2^\Gamma)^M$ be the Bernoulli shift of $\Gamma$, as computed in $M$. Then, $s^M(\gamma,x)=y$ if and only if $y(\delta)=(\gamma^{-1}\delta)$ for all $\delta\in\Gamma$, which is arithmetic in the multiplication table of $\Gamma$. Thus, we can express that $s$ is the Bernoulli shift of $\Gamma$ as a $\Pi^1_1$ predicate in a code for $s$, using a real which codes $\Gamma$ and its multiplication operation.
	
	(ii) Likewise, the definition of the free part of the Bernoulli shift is also arithmetic in the multiplication table of $\Gamma$.
\end{proof}

\begin{thm}\label{thm:randoms_shift}
	Suppose that $\Gamma$ is a countably infinite group in $M$. Identifying $2^\omega$ with $2^\Gamma$, $\LR$ is a $\Gamma$-invariant conull Borel subset of the free part of the Bernoulli shift action of $\Gamma$ and $F(\Gamma,2)\restr\LR\;\subseteq\;\equiv^\B_M$.
\end{thm}

\begin{proof}
	We identify (in $M$ and $\V$) $2^\Gamma$ with $2^\omega$ via some fixed bijection $\Gamma\to\omega$ in $M$, and thus $\B$ becomes the set of all Borel subsets of $2^\Gamma$ having positive measure in $M$, and $\LR$ a Borel subset of $2^\Gamma$ in $\V$.

	By Lemma \ref{lem:Bern_shift_abs}(i), the shift action of each $\gamma\in\Gamma$ is a measure-preserving Borel bijection of $2^\Gamma$ coded in $M$.
	
	We claim that $\LR$ is $\Gamma$-invariant, that is, for every $x\in 2^\Gamma$ and $\gamma\in \Gamma$, if $x$ is random over $M$, then so is $\gamma\cdot x$. To this end, suppose that $x\in\LR$ and let $\gamma\in \Gamma$. Let $H_x$ be the corresponding $M$-generic filter for $\B$. By Lemma \ref{lem:map_Bgens}, $(\gamma^M)\cdot H_x=H_{\gamma\cdot x}$, and so $\gamma\cdot x$ is the random real corresponding to the $M$-generic filter $\gamma^M\cdot H_x$.
	
	Observe that if $x\in\LR$ and $\gamma\in\Gamma$, then Lemma \ref{lem:map_Bgens} also shows that $M[x]=M[\gamma\cdot x]$. Thus, the restriction of $E(\Gamma,2)\restr\LR$ is contained in $\equiv^\B_M$.
	
	Lastly, we claim that $\LR$ is contained in the free part the Bernoulli shift action, $B=(2)^\Gamma$. By Lemma \ref{lem:Bern_shift_abs}(ii), $B$ is coded correctly in $M$, and so $B^M$ is conull in $M$ by the absoluteness of Lebesgue measure (Lemma II.1.6.4 in \cite{MR0265151}). As a real is in $\LR$ exactly when it is contained in every conull Borel set coded in $M$ (Theorem II.2.6 in \cite{MR0265151}), we have that $\LR\subseteq B=(2)^\Gamma$. Note that $\LR$ itself is conull, again by the absoluteness of Lebesgue measure.
\end{proof}

\begin{thm}\label{thm:random_not_amen}
	$\equiv^\B_M$ is not amenable. In particular, $\equiv^\B_M$ is not hyperfinite.	
\end{thm}

\begin{proof}
	Take $\Gamma=F_2$, which is clearly in $M$. By Theorem \ref{thm:randoms_shift}, $\Gamma$ acts in a free, measure preserving way on the conull set $\LR\subseteq 2^\Gamma$ of $M$-generic random reals. Since $\Gamma$ is a non-amenable group, any free, measure-preserving Borel action of $\Gamma$ on a standard probability space yields a non-amenable orbit equivalence relation (Proposition 2.14 in \cite{MR1900547}). Hence, $F(\Gamma,2)\restr\LR$ is not amenable. Since amenability is inherited by subequivalence relations (Proposition 2.15 in \cite{MR1900547}), $\equiv^\B_M$ is not amenable, and thus not hyperfinite, either.
\end{proof}


Amenability is inherited downwards via Borel reductions (Proposition 2.15 in \cite{MR1900547}), so Theorems \ref{thm:cohen_hhf} and \ref{thm:random_not_amen} imply:

\begin{cor}\label{cor:random_nleq_cohen}
	$\equiv^\B_M\;\not\leq_B\;\equiv^\C_M$.\qed
\end{cor}

Next, we will make use of results for actions of Kazdahn groups. We will not define the Kazdahn property here (see \cite{MR776417}), as for our applications, only one such group is needed, namely $\SL_3(\Z)$. 

A Borel homomorphism $f$ from equivalence relations $E$ to $F$, where $E$ lives on a standard measure space $(X,\mu)$, is \emph{$\mu$-trivial} if there is a $\mu$-conull Borel set $C\subseteq X$ which $f$ maps to a single $F$-class. If no such set exists, it is \emph{$\mu$-nontrivial}. The following theorem is a consequence of results in \cite{MR1047300}:

\begin{thm}[Hjorth--Kechris, Theorem 10.5 in \cite{MR1423420}]\label{thm:HK_not_treeable}
	Let $\Gamma$ be a countable Kazhdan group which acts in an ergodic measure-preserving Borel way on a standard probability space $(X,\mu)$. If $F$ is a treeable countable Borel equivalence relation, then every Borel homomorphism from $E$ to $F$ is $\mu$-trivial.
\end{thm}

\begin{thm}\label{thm:randoms_not_treeable}
	$\equiv^\B_M$ is not treeable.
\end{thm}

\begin{proof}
	Let $\Gamma=\SL_3(\Z)$, which is clearly in $M$. Theorem \ref{thm:HK_not_treeable} implies that $F(\Gamma,2)\restr\LR$ is not treeable, where $\LR$ is the conull set of $M$-generic random reals in $2^\Gamma$. By Theorem \ref{thm:randoms_shift}, $F(\Gamma,2)\restr\LR\subseteq\;\equiv^\B_M$, and since treeability is inherited by subequivalence relations (Proposition 3.3 in \cite{MR1900547}), $\equiv^\B_M$ is not treeable either.
\end{proof}

The next theorem we need is an application of results from \cite{MR2342637}:

\begin{thm}[Thomas, Theorem 3.6 in \cite{MR2500091}]\label{thm:thomas_not_free}
	Let $\Delta = \SL_3(\Z)\times S$, where $S$ is any countable group. Suppose that $\Gamma$ is a countable group that acts in a free Borel way on a standard Borel space, with orbit equivalence relation $E$. If there exists a $\mu$-nontrivial Borel homomorphism from $E(\Delta,2)$ to $E$, then there exists a group homorphism $\pi :\Delta\to\Gamma$ with finite kernel.
\end{thm}

The proof of the following is modeled on that of Theorem 3.9 in \cite{MR2500091}.

\begin{thm}\label{thm:randoms_not_Mfree}
	$\equiv^\B_M$ is not Borel reducible to the orbit equivalence relation induced by a free action of any countable group $\Gamma$ in $M$.
\end{thm}

\begin{proof}
	Let $\Gamma$ be a countably infinite group in $M$. Suppose, towards a contradiction, that $\equiv^\B_M$ is Borel reducible to the orbit equivalence relation $E$ of a free Borel action of $\Gamma$ on some standard Borel space. We remark that $E$ need not be coded in $M$.
	
	Working in $M$, as there are uncountably many finitely generated groups, there exists a finitely generated group $L$ which does not embed into $\Gamma$. Let $S$ be the free product $L\ast\Z$ and let $\Delta=\SL_3(\Z)\times S$. Then, $\Delta$ has no nontrivial finite normal subgroups and does not embed into $\Gamma$. The non-existence of such an embedding is absolute and thus applies in $\V$.
	
	By Theorem \ref{thm:randoms_shift}, $F(\Delta,2)\restr\LR\subseteq\;\equiv^\B_M$, where $\LR$ is the conull set of $M$-generic random reals in $2^\Delta$. This induces a $\mu$-nontrivial homomorphism from $F(\Delta,2)$ to $E$, and thus, by Theorem \ref{thm:thomas_not_free}, an embedding of $\Delta$ into $\Gamma$, a contradiction.
\end{proof}

We note that Theorem \ref{thm:randoms_not_Mfree} implies Theorem \ref{thm:randoms_not_treeable}, since $F(F_2,2)$ is universal for all treeable Borel equivalence relations (Theorem 3.17 in \cite{MR1900547}).


Lastly, we turn to the question of whether $\equiv^\B_M$ is a universal countable Borel equivalence relation. We will employ two results about universality that are conditional on the following conjecture of Andrew Marks:

\begin{marks_conj}[Conjecture 1.4 in \cite{MR3651212}]
	A countable Borel equivalence relation is universal if and only if it is uniformly universal with respect to every way it can be generated.	
\end{marks_conj}

See \cite{MR3651212} for the relevant definitions and details. We remark that this conjecture is closely related to Martin's conjecture on Turing invariant Borel maps; for instance, Theorem \ref{thm:marks_univ}(2) below is also a consequence of Martin's conjecture (Theorem 5.4 in \cite{MR2563815}).

\begin{thm}[Marks, cf.~Theorem 1.5 in \cite{MR3651212}]\label{thm:marks_univ}
	Assume Marks's uniformity conjecture.
	\begin{enumerate}[label=\textup{(\arabic*)}]
		\item An increasing union of countably many non-universal countable Borel equivalence relations is not universal.
		\item If $E$ is a countable Borel equivalence relation on a standard probability space $(X,\mu)$, then there is a $\mu$-conull Borel set $Z\subseteq X$ for which $E\restr Z$ is not universal.
	\end{enumerate}
\end{thm}

A model $M$ of $\ZFC$ is \emph{$\mathbf{\Sigma}^1_2$-correct} if $\Sigma^1_2$ formulas with parameters in $M$ are absolute. While a countable transitive model of $\ZFC$ may failto be $\mathbf{\Sigma}^1_2$-correct, $\mathbf{\Sigma}^1_2$-correct countable models are plentiful under mild assumptions, e.g., take the transitive collapse of a countable elementary submodel of $\V_\kappa$, when $\kappa$ is inaccessible.\footnote{We would like to thank Gabriel Goldberg for pointing this out.} \footnote{As per footnote 1, we may allow $\omega_1\subseteq M$, provided $|\LP(\B)\cap M|=|\LP(\R)\cap M|$ is still countable in $\V$. In this case, $M$ will be $\mathbf{\Sigma}^1_2$-correct by Shoenfield's Absoluteness Theorem (Theorem 25.20 in \cite{MR1940513}), without any additional hypotheses.} 

\begin{thm}\label{thm:randoms_not_univ}
	Assume Mark's uniformity conjecture. If $M$ is $\mathbf{\Sigma}^1_2$-correct, then $\equiv^\B_M$ is not universal.
\end{thm}

To prove this result, we first observe the following lemma, a consequence of the universality of $E(F_2,2)$ and quantifier counting.

\begin{lemma}\label{lem:univ_sig12}
	The statement that a countable Borel equivalence relation $E$ is universal is $\Sigma^1_2$ in a code for $E$.\qed
\end{lemma}

We will also need the measure-theoretic analogue of Theorem \ref{thm:realize_auts}:

\begin{thm}[von Neumann \cite{MR1503077}; Theorem 15.21 in \cite{MR1013117}]\label{thm:realize_meas_auts}
	Every automorphism of the Boolean algebra of Borel sets modulo null sets, in a standard probability space $(X,\mu)$, is induced by a measure-preserving Borel bijection $f:X\to X$. 
\end{thm}

\begin{proof}[Proof of Theorem \ref{thm:randoms_not_univ}]
	This argument is very similar to that in the proof of Theorem \ref{thm:cohen_hhf}, so we will omit some of the details. Enumerate (in $\V$) $\Aut^M(\bar{\B})$ as $\{\gamma_n:n\in\omega\}$. For each $n$, in $M$ we use Theorem \ref{thm:realize_meas_auts} to find a measure-preserving Borel bijection $f_n^M$ of $(2^\omega)^M$ which induces the automorphism $\gamma_n$. Let $\Gamma_n$ be an abstract group in $M$ which is isomorphic to that generated by $\{f_0^M,\ldots,f_{n-1}^M\}$, and let $E_n^M$ be the induced orbit equivalence relation on $(2^\omega)^M$.
		
	For each $n$, apply Theorem \ref{thm:marks_univ}(2) in $M$ to obtain a conull Borel set $C_n^M$ on which $E_n^M$ is not universal (in $M$).
	
	In $\V$, let $C=\bigcap_{n\in\omega} C_n$, and let $E=\bigcup_{n\in\omega} (E_n\restr C)$. Since $M$ is $\mathbf{\Sigma}^1_2$-correct, each $E_n$ is not universal in $\V$ by Lemma \ref{lem:univ_sig12}. Thus, by Theorem \ref{thm:marks_univ}(1), $E$ is not universal. Since $M$-generic random reals are contained in every conull Borel set coded in $M$, $\LR\subseteq C$.
	
	It remains to argue that $\equiv^\B_M$ coincides with $E\restr\LR$. This is done in exactly the same way as for $\equiv^\C_M$ in the proof of Theorem \ref{thm:cohen_hhf}, using the completeness (in $M$) of the Boolean algebra of Borel sets modulo null sets, Theorem \ref{thm:bool_auts}, and Lemmas \ref{lem:eqrel_abs} and \ref{lem:map_Bgens}.
\end{proof}

\section{Further questions}\label{sec:questions}

Much of this paper has focused on Cohen and random forcing. This is due to the intimate connections these examples have with Baire category and measure, respectively. A natural next step is to consider those partial orders mentioned in Example \ref{ex:other_posets}:

\begin{ques}
	What can we say about the complexity of $\equiv^\P_M$ when $\P$ is, e.g., Sacks, Miller, Mathias, or Laver forcing, or any of the other classical forcing notions? In particular, is $\equiv^\P_M$ hyperfinite in any of these cases?
\end{ques}

The increasing unions problem for hyperfinite equivalence relations \cite{MR1149121} asks whether an increasing union of countably many hyperfinite Borel equivalence relations is hyperfinite. By Theorem \ref{thm:cohen_hhf}, a positive resolution to this problem would imply that $\equiv^\C_M$ is hyperfinite, when $\C$ is Cohen forcing. This provides an interesting special case of this long-standing open problem:

\begin{ques}\label{ques:cohen_hf}
	Is $\equiv^\C_M$ hyperfinite?
\end{ques}

Given a positive answer to Question \ref{ques:cohen_hf}, the Glimm-Effros dichotomy for countable Borel equivalence relations (Theorem 1.5 in \cite{MR1802331}) would allow us to improve Corollary \ref{cor:random_nleq_cohen} to $\equiv^\C_M\;<_B\;\equiv^\B_M$.



In Theorem \ref{thm:randoms_not_Mfree}, we left open the question of whether $\equiv^\B_M$, for $\B$ random forcing, is \emph{essentially free}, that is, reducible to the orbit equivalence relation of a free Borel action of some countable group. Note that it is easy to check that the associated action of $\Aut^M(\B)$ on the random reals is not free.

\begin{ques}\label{ques:random_not_free}
	Is $\equiv^\B_M$ essentially free?	
\end{ques}

Since there is no countable collection of groups such that an equivalence relation is essentially free if and only if it is reducible to an orbit equivalence relation coming from a free action of a group in that collection (cf.~Corollary 3.10 in \cite{MR2500091}), Theorem \ref{thm:randoms_not_Mfree} cannot be used directly to resolve Question \ref{ques:random_not_free}.

Can we obtain the result in Theorem \ref{thm:randoms_not_univ} without the extra hypotheses of Marks's uniformity conjecture and the $\mathbf{\Sigma}^1_2$-correctness of $M$?

\begin{ques}
	Is $\equiv^\B_M$ non-universal (for any $M$)?
\end{ques}

More generally, to what extent does the model $M$ affect the complexity of $\equiv^\P_M$? In asking this question, we have to avoid certain trivialities; for instance, by collapsing $|\P|^M$ to $\aleph_0$, we can go from a model $M$ for which $\equiv^\P_M$ is smooth (Example \ref{ex:rigid}), to a model $M'$ in which $\P$ is equivalent to Cohen forcing and thus $\equiv^{\P}_{M'}$ is not smooth. So, we focus on the case of idealized forcings which are $\ZFC$-correct, in the sense of \cite{MR2391923}, and thus have a natural interpretation in any model.

\begin{ques}
	Is there a $\ZFC$-correct $\sigma$-ideal $I$ on $2^\omega$ and countable transitive models $M$ and $N$ such that $\equiv^{(P_I)^M}_M$ and $\equiv^{(P_I)^N}_N$ are not Borel bireducible?
\end{ques}

\bibliography{/Users/iian/Dropbox/Mathematics/math_bib}{}
\bibliographystyle{abbrv}

\end{document}